\documentclass[11pt,leqno]{amsart}
\usepackage[utf8]{inputenc}
\usepackage{graphicx}
\usepackage{siunitx}
\usepackage{amsmath,amsthm,amsfonts,amssymb}
\usepackage{mathtools}
\usepackage{caption}
\usepackage[margin=1.20in]{geometry}
\usepackage{float}
\usepackage{hyperref}

\DeclareMathAlphabet\mathbfcal{OMS}{cmsy}{b}{n}

\newtheorem{theorem}{Theorem}[section]
\newtheorem{lemma}[theorem]{Lemma}
\newtheorem{remark}[theorem]{Remark}
\newtheorem{exmp}[theorem]{Example}
\newtheorem{corollary}[theorem]{Corollary}
\newtheorem{prop}[theorem]{Proposition}
\theoremstyle{definition}
\newtheorem{definition}[theorem]{Definition}
\newtheorem{case}{Case}
\newtheorem{subcase}{Case}
\numberwithin{subcase}{case}

\graphicspath{{images/}}

\title{Interpolating classical partitions of the set of positive integers}
\author[Weiru Chen]{Weiru Chen}
\email[Weiru Chen]{wchen108@illinois.edu}
\author[Jared Krandel]{Jared Krandel}
\email[Jared Krandel]{jaredjk2@illinois.edu}
\date{}

\begin{document}
\maketitle

\begin{abstract}

We construct an easily described family of partitions of the positive integers into $n$ disjoint sets with essentially the same structure for every $n \geq 2$. In a special case, it interpolates between the Beatty $\frac{1}{\phi} + \frac{1}{\phi^2} = 1$ partitioning ($n=2$) and the 2-adic partitioning in the limit as $n \rightarrow \infty$. We then analyze how membership of elements in the sets of one partition relates to membership in the sets of another. We investigate in detail the interactions of two Beatty partitions with one another and the interactions of the $\phi$ Beatty partition mentioned above with its ``extension'' to three sets given by the construction detailed in the first part. In the first case, we obtain detailed results whereas the second case we place some restrictions on the interaction but cannot obtain exhaustive results.

\end{abstract}

\newpage

\section{Introduction}

Beatty's $\frac{1}{x} + \frac{1}{y} = 1$ theorem \cite{B} provides uncountably many partitionings of the set of positive integers. Among these are partitionings that encode winning strategies for variations of Wythoff's game \cite{W}. Particularly relevant is the partition given by 
\begin{equation*}
    \frac{1}{\phi} + \frac{1}{\phi^2} = 1
\end{equation*}
with the members of the two distinct sequences 
\begin{equation*}
    a(k) = \lfloor k\phi \rfloor \quad\text{and}\quad  b(k) = \lfloor k\phi^2 \rfloor
\end{equation*}
giving the two sets of the partition. Another significant partition, this time into infinitely many sets, is the 2-adic partition. Here, the sets are $2^e\cdot M$ for $e = 0, 1, 2, \ldots, $ where $M$ is the set of all positive odd integers. This partitioning is, for example, the key to showing that the partial sums of the harmonic series are (with the trivial exception) never integral. See [p. 32, problem 36] of \cite{N-Z-M}.

We construct an easily described family of partitions of the positive integers into $n$ disjoint sets with essentially the same structure for every $n \geq 2$. In a special case, it interpolates between the Beatty golden ratio partitioning $(n = 2)$ and the 2-adic partitioning in the limit as $n\rightarrow\infty$. 

Next, we study how specific examples of partitions in the constructed family relate to one another. If we define $A = \{ \lfloor k\phi \rfloor \}_{k=1}^\infty$ and $B = \{ \lfloor k\phi^2 \rfloor \}_{k=1}^\infty$, then the positive integers can be described as $A\cup B$ or $D_1\cup D_2 \cup D_3$ where the former is an $n=2$ partition and the latter is an $n=3$ partition to be described (see Section 2.4) which ``extends'' the $A,B$ partition to 3 sets. In order to carry out this analysis, we write $D_j$ as a sequence $\{d_{1j}, d_{2j}, d_{3j}, \ldots\}$ with $d_{ij} < d_{(i+1)j},\ 1\leq\ j \leq 3$. Then the $n=3$ partitioning corresponds to the 3 infinite columns of an $\infty \times 3$ matrix shown in Table 1. 
\begin{table}[h]
    \centering
    \begin{tabular}{c|c|c}
        $D_3$ & $D_2$ & $D_1$ \\ \hline
        1 & 2 & 4 \\
        3 & 6 & 11 \\
        5 & 9 & 15 \\
        7 & 13 & 22 \\
        8 & 17 & 29 \\
        10 & 20 & 33 \\
        \vdots & \vdots & \vdots
    \end{tabular}
    \caption{An enumeration of the first few values of $D_1,D_2,$ and $D_3$}
    \label{tab:my_label}
\end{table}
We now take the transposed point of view and examine the rows of this matrix. They form an infinite collection of 3-vectors $(d_{k1}, d_{k2}, d_{k3})$. Each such 3-vector has, a priori, $2^3 = 8$ possible classifications depending on the $A$ or $B$ membership of its individual components. For example, $(d_{11}, d_{12}, d_{13}) = (1,2,4)\in A\times B \times A$ while $(d_{21}, d_{22}, d_{23}) = (3,6,11)\in A\times A\times A$; these are two of the 8 distinct possibilities. We perform a detailed technical analysis to determine which of these cases occur. The result seems non-obvious: exactly 6 of the 8 cases can occur. This leads to a natural question that we have been unable to resolve: With what frequencies do each of these 6 cases occur?

We preface the above investigation by a similar but more tractable problem. Consider the Beatty partitioning associated to
\begin{equation*}
    \frac{1}{\phi^3} + \frac{1}{\phi^2/2} = 1,
\end{equation*}
and let $C$ and $D$ be the two sets in this partition. One can ask how each of $C$ and $D$ splits into $A$ and $B$ elements and vice versa. In this case, we derive a summary of fractional parts identities which gives fairly complete information about the problem. We then extend this to a more generalized case involving the Fibonacci numbers that can be used to further investigate such classification problems. This Beatty partitioning case is of special interest, as the Beatty sequence $\lfloor n\phi^3 \rfloor$ plays a role (``Long's conjecture'') in the theory of additive partitions, a type of partitioning that is not of Beatty type, but is closely related to Beatty-type partitions. Specifically, it is of interest with respect to the unique partitioning of the positive integers into disjoint sets $A^*$ and $B^*$ such that sums of distinct elements of $A^*$, and also of $B^*$, never equal a Fibonacci number. See \cite{C} for the proof of Long's conjecture and \cite{C-L} for further background information.

Two main tools for our investigations are the KLM formula (see \cite{F-P-S}[p. 256] for a proof) and Theorem 3.2, a proof of which is available in \cite{P-S}. For general information on Beatty sequences and especially explanations as to why the naive generalization of Beatty's theorem to sums of 3 or more irrational reciprocals is false (Uspensky's theorem) see \cite{F}, \cite{N}, and \cite{U}. Interesting ways of partitioning the set of all real numbers into $n$ disjoint subsets are described in \cite{H-Z}. 

\section{n-Set Partition Theorem}

\subsection{Statement of Theorem and Examples}

\begin{definition}

Let $r_i, i = 1,2,3,\ldots,$ be real numbers and $S$ a set of real numbers. Define
\begin{align*}
    S \pm r_1 &= \{s+r_1:s \in S \} \cup \{s-r_1:s \in S \},\\
    S \pm r_1 \pm r_2 &= (S \pm r_1) \pm r_2,\\
    \text{and } S \pm r_1 \pm \ldots \pm r_m &= (S \pm r_1 \pm \ldots \pm r_{m-1}) \pm r_m.
\end{align*}

\end{definition}

\vspace{\baselineskip}

\begin{theorem}

Let $n \geq 2$ and define 
\[
    G = \{2^{n-1}, 2^{n-1} + 2^{n-2},\ldots,\sum_{i=1}^{n}2^{n-i}\}.
\]
Let $l$ be a non-decreasing integer-valued sequence such that $l(1) = 2^{n-1}$ and $l(k+1) - l(k) \in G$ for all $k$. If we set $D_1 = \{l(k)\}_{k=1}^{\infty}$, then the sets
$$D_1, D_2 = D_1 \pm 2^{n-2}, D_3 = D_1 \pm 2^{n-2} \pm 2^{n-3}, \ldots, D_n = D_1 \pm 2^{n-2} \pm 2^{n-3} \pm \ldots \pm 2 \pm 1 $$
are a partition of the set of all positive integers into $n$ disjoint sets.

\end{theorem}

\begin{corollary}
Let $h(k)$ be a non-decreasing integer-valued sequence with $h(1)=1$ and \\ $1 \leq h(k+1) - h(k) \leq 2$. Put $l(k) = t(k) = (2^{n-1} - 1)h(k) + k$ for $n \geq 2$.Then the sets $D_1,D_2,\ldots, D_n$ defined in Theorem 2.2 with $D_1 = \{t(k)\}_{k=1}^{\infty}$ partition the positive integers. 
\end{corollary}

\begin{proof}
We need to check that $t(1) = 2^{n-1}$ and $t(k+1)-t(k)$ takes on values in $G$. The former is clear from the definition of $t$. To prove the latter, consider the two cases where $h(k+1) - h(k) = 1$ and  where $h(k+1) - h(k) = 2$. In the former case, we have
\begin{align*}
    t(k+1) - t(k) &= (2^{n-1}-1)(h(k+1)-h(k)) + 1 \\
    &= (2^{n-1} - 1) + 1 = 2^{n-1}
\end{align*}
which is the first number in $G$. In the second case, the same calculation gives
\begin{align*}
    t(k+1) - t(k) &= (2^{n-1} - 1)2 - 1\\
    &= 2^n - 1 = \sum_{i=1}^n2^{n-i}
\end{align*}
which is the final number in $G$.
\end{proof}

\begin{exmp}
Let $t(k)$ be as in Corollary 2.3 and set $h(k) = k$ and $n=2$. This simply gives the partitioning of the positive integers into even and odd numbers. For $n=3$ it gives
$$D_1 = \{4k+4\}, $$
$$D_2 = \{4k+2\}, $$
$$D_3 = \{2k+1\}, $$
where $k = 0,1,2,\ldots$
\end{exmp}

\begin{remark}
Our partitioning into $n$ sets is of interest in part because of Uspensky's Theorem , which states that the sequences $\{\lfloor kx_1 \rfloor \}, \{\lfloor kx_2 \rfloor \}, \ldots, \{\lfloor kx_n \rfloor \}$ where $x_1, x_2, \ldots x_n$ are irrational numbers satisfying
$$\frac{1}{x_1} + \frac{1}{x_2} + \ldots + \frac{1}{x_n} = 1$$
give a partitioning of the positive integers only when $n=2$.
\end{remark}

\begin{corollary}
Let $1 \leq \alpha < 2$ and let $n\in\mathbb{N}$ with $n\geq 2$. Put $t(k) = (2^{n-1}-1)\lfloor k\alpha \rfloor + k$ for $k\in\mathbb{N}$. Then the sets $D_1,D_2,\ldots D_n$ as defined in Theorem 2.2 with $D_1 = \{t(k)\}_{k=1}^\infty$ partition the positive integers.
\end{corollary}

\begin{proof}
Clearly $t(1) = 2^{n-1}$. Since
\begin{align*}
    \lfloor (k+1)\alpha \rfloor - \lfloor k\alpha \rfloor &= (k+1)\alpha - \{(k+1)\alpha \} - (k\alpha - \{k\alpha \})\\
    &= \alpha + \{k\alpha \} - \{(k+1)\alpha \}\\
    &< 2 + \beta
\end{align*}
where $\beta < 1$, so the difference is strictly less than $3$. Since the difference is an integer, it must be at most 2, and the result follows from Corollary 2.3.
\end{proof}

\begin{exmp}
Let $h(k) = \lfloor k\phi \rfloor$, where $\phi$ is the golden ratio, and $n=2$. Then $D_1 = \{\lfloor (2^1-1)\lfloor k\phi \rfloor + k\} = \{\lfloor k(\phi+1) \rfloor\}_{k=1}^\infty = \{\lfloor k\phi^2 \rfloor\}_{k=1}^\infty$ because $\phi^2 = \phi + 1$. This implies $D_2 = \{ \lfloor k\phi \rfloor\}_{k=1}^\infty$, so $D_1$ and $D_2$ form the $\frac{1}{\phi} + \frac{1}{\phi^2}$ Beatty partition.
\end{exmp}

\begin{exmp}
Let $h(k) = \lfloor k\phi \rfloor $ and $n=3$. Then we have a possibly new partitioning given by
\begin{align*}
    D_1 &= \{3a(k) + k\} = \{ 4, 11, 15, 22, 29, 33, \ldots \}\\
    D_2 &= \{a(k) + 2k - 1\} = \{ 2, 6, 9, 13, 17, 20, 24, 27, \ldots \}\\
    D_3 &= \{\mathbb{N}/(D_1 \cup D_2)\} = \{ 1, 3, 5, 7, 8, 10, 12, 14, 16, \ldots \}
\end{align*}
The set $D_2$ does occur in OEIS as sequence A054770. It is of interest in connection with properties of Lucas numbers. This is the aforementioned $n=3$ case ``extension'' of the $A,B$ partition to be studied in later sections.
\end{exmp}

\begin{exmp}
Let $h(k) = \lfloor k\sqrt{2} \rfloor $ and $n=2$. This gives 
\begin{align*}
D_1 &= \{2, 4, 7, 9, 12, \ldots \},\\
D_2 &= \{1, 3, 5, 6, 8, 10, 11, \ldots \}.
\end{align*}
Unlike in the golden ratio case, this is not the Beatty $\frac{1}{\sqrt{2}} + \frac{1}{2+\sqrt{2}} = 1$ partition which consists of the sets
\begin{align*}
S_1 &= \{3, 6, 10, 13, 17, \ldots \}, \\
S_2 &= \{1, 2, 4, 5, 7, 8, 9, 11, \ldots \}.
\end{align*}
\end{exmp}

\subsection{Proof of Theorem 2.2}

The proof will be split into two parts. First, we show that the sets $D_1, D_2, \ldots, D_n$ contain every positive integer. Second, we show that these sets are pairwise disjoint, i.e., $D_i \cap D_j = \varnothing,\ i \not= j$. We begin by developing a series of lemmas from which the first part of the proof will follow easily.

\vspace{\baselineskip}

\begin{lemma}

Suppose that $a_i, b_i, c, n, N,$ and $m$ are integers with $n \geq c \geq 0, N \geq n+m, \text{ and } m \geq 1.$ If the $b_i$ are odd, then
$$\sum_{i=c}^{n} a_i\cdot 2^{N-i} \not= \sum_{i=c}^{n+m} b_i \cdot 2^{N-i}.$$

\end{lemma}

\begin{proof}
Suppose equality holds. Then
$$\sum_{i=c}^{n} a_i\cdot 2^{n+m-i} = \sum_{i=c}^{n+m} b_i \cdot 2^{n+m-i},$$
and all terms in these sums are integers. Since $m \geq 1$, all terms on the left side of the equation are even. On the other hand, all terms on the right side of the equation are even except for the last term which equals $b_{n+m}$ and is odd. This is a contradiction.
\end{proof}

\begin{definition}
Let $t$ be an integer and let the $(n-1)$-vector $E$ be defined by
$$E = E_{n-1} = \{\epsilon _0, \epsilon _1, \ldots , \epsilon _{n-2}\}$$ 
where we make the restriction $\epsilon_{j}\in\{-1,1\}$ for the remainder of the paper. Define $n$ inhomogeneous linear forms $L_0 = L_{0}(t,E), \ldots, L_{n-1} = L_{n-1}(t,E)$ in the variables $\{t; \epsilon _0, \epsilon _1, \ldots , \epsilon _{n-2} \} $ by
\begin{align*}
     L_0(t,E) &= t,\\
L_1(t,E) &= t + \epsilon _{n-2} \cdot 2^{n-2},\\
L_2(t,E) &= t + \epsilon _{n-2} \cdot 2^{n-2} + \epsilon _{n-3} \cdot 2^{n-3},\\
\vdots&\\
L_j(t,E) &= t + \epsilon _{n-2} \cdot 2^{n-2} + \ldots + \epsilon _{n-j-1} \cdot 2^{n-j-1},\\
\vdots&\\
L_{n-1}(t,E) &= t + \epsilon _{n-2} \cdot 2^{n-2} + \ldots + \epsilon _{1} \cdot 2 + \epsilon _0.
\end{align*}
Thus, for $j\geq 1$,
$$L_j(t,E) = t + \sum_{s=2}^{j+1} \epsilon_{n-s} \cdot 2^{n-s}.$$
\end{definition}

\begin{lemma}
Fix $t\in\mathbb{N}$. As $E$ varies over all of its $2^{n-1}$ possible values, the above collection of linear forms $\{L_0, L_1, \ldots , L_{n-1} \}$ takes on each value in the closed interval of integers
$$I(t) = [ t-(2^{n-1}-1), t+(2^{n-1}-1) ]$$
and no other values. Moreover, each integer in $I(t)$ is the value of exactly one of the linear forms as $E$ varies over all $2^{n-1}$ possible values.
\end{lemma}

\begin{proof}

The linear forms $\{L_0, L_1, \ldots , L_{n-1} \} $ assume respectively 
$$1, 2, 2^2, \ldots , 2^{n-1}$$
possible values, hence $\sum_{i=0}^{n-1} 2^i = 2^n -1$ possible values in all. Let $i,j$ be distinct integers between 0 and $n-1$. By Lemma 2.10, we have that $L_i(t,E_1)\not=L_j(t,E_2)$ for any $E_1,E_2$. In addition, we claim that for any $j$, $L_j(t,E_1)\not=L_j(t,E_2)$. That is, for fixed $t$, no two values taken on by $L_j(t)$ as it varies over $E$ are the same. For proof, suppose two such values are equal. Then there exist a collection $\{\epsilon _0, \epsilon _1, \ldots , \epsilon _{n-j-1}\}$ and a distinct collection $\{\delta _0, \delta _1, \ldots , \delta _{n-j-1}\}$ each taking values in $\{-1,1\}$ such that 
$$t + \sum_{s=2}^{j+1} \epsilon _{n-s} \cdot 2^{n-s} = t + \sum_{s=2}^{j+1} \delta _{n-s} \cdot 2^{n-s}.$$
Thus
$$\sum_{s=2}^{j+1} (\epsilon _{n-s}-\delta _{n-s}) \cdot 2^{n-s} = 0.$$
Since $\epsilon _{n-s}-\delta _{n-s}$ must be -2, 0, or 2, we can divide the equation by $2^{n-j}$ to get
$$\sum_{s=2}^{j+1} \frac{(\epsilon _{n-s}-\delta _{n-s})}{2} \cdot 2^{j+1-s} = 0.$$
If $\epsilon _{n-j-1}-\delta _{n-j-1}\not=0$, then all the terms of the equation are even except for the last term, which equals 1 or -1 and is odd. This is a contradiction. So $\epsilon _{n-j-1}-\delta _{n-j-1}$ must be 0. Now we have
$$\sum_{s=2}^{j} (\epsilon _{n-s}-\delta _{n-s}) \cdot 2^{n-s} = 0.$$
Repeating this process, we obtain $\epsilon _{n-s}-\delta _{n-s} = 0$ for all $s \in \{2, 3, \ldots , j+1 \}$. So the collections $\{\epsilon _0, \epsilon _1, \ldots , \epsilon _{n-j-1}\}$ and $\{\delta _0, \delta _1, \ldots , \delta _{n-j-1}\}$ are exactly the same. Therefore, for fixed $t\in\mathbb{N}$, the form $L_j(t)$ is injective for any $j \in \mathbb{N}$. Hence the linear forms $\{L_0, L_1, \ldots , L_{n-1} \} $ assume all possible
$$(2^{n-1}-1) + 1 + (2^{n-1}-1) = 2^n-1$$
values between their minimum and maximum values. These values are
$$t \pm \sum_{i=0}^{n-2} 2^i = t \pm (2^{n-1}-1),$$
and the result follows.
\end{proof}

\begin{remark}

We shall next allow $t$ to vary. But it will be useful to bear in mind that if at any point we fix $t$, distinct forms will take on distinct values.

\end{remark}

\begin{lemma}

Let $l$ be a sequence satisfying the hypotheses of Theorem 2.2 and set 
\begin{equation*}
    I(l(k)) = [l(k)-(2^{n-1}-1), l(k)+(2^{n-1}-1)]. 
\end{equation*}
Then $\bigcup \limits_{k=1}^{\infty} I(l(k))$ is the set of all positive integers.

\end{lemma}

\begin{proof}
By the hypotheses of Theorem 2.2, $l(k) \rightarrow \infty$ as $k \rightarrow \infty$, and $l(1)-(2^{n-1}-1) = 1$. Thus, we need to show that for $k\geq 2$, any given interval $I(l(k))$ either overlaps with $I(l(k-1))$ or begins at the first integer outside $I(l(k-1))$. This condition will be satisfied if
\begin{equation*}
    l(k+1)-(2^{n-1}-1) \leq l(k)+(2^{n-1}-1)+1.
\end{equation*}
Rearranging this inequality gives
\begin{equation*}
    l(k+1)-l(k) \leq 2^n - 1.
\end{equation*}
Since
\[
    \text{max}\ G = \sum_{i=1}^{n}2^{n-i} = 2^n - 1,
\]
this inequality is satisfied for all $k$.
\end{proof}

\begin{corollary}

Let $t$ be as in Corollary 2.3 and set $h(k) = \lfloor k\alpha \rfloor$ where $1 \leq \alpha < 2$. Then $\bigcup \limits_{k=1}^{\infty} I(t(k))$ is the set of all positive integers.

\end{corollary}

\begin{prop}
Let the sequence $l$ and the sets $D_j, 1 \leq j \leq n,$ be as in Theorem 2.2 with $D_1 = \{l(k)\}_{k=1}^{\infty}$. Then $\bigcup \limits_{j=1}^{n} D_j$ is the set of all positive integers.
\end{prop}

\begin{proof}
First, fix $k\in \mathbb{N}$ and let the set of all values of the form $L_{j-1}(l(k),E)$ as $E$ varies varies over all possible choices for $\epsilon_0,\epsilon_1,\ldots,\epsilon_{n-2}$ be denoted by
\begin{equation*}
    \bigcup\limits_{E}\{L_{j-1}(l(k),E)\}.
\end{equation*}
Each $D_j$ is defined such that it is the set of all values taken on by $L_{j-1}(l(k),E)$ as $k$ varies over all positive integers and $E$ varies as above.
With this characterization of $D_1,D_2,\ldots D_n$, we can write
\begin{equation*}
    \bigcup \limits_{j=1}^{n} D_j = \bigcup \limits_{j=1}^{n}\bigcup\limits_{k=1}^\infty \left(\bigcup\limits_{E} \{L_{j-1}(l(k),E)\}\right) = \bigcup\limits_{k=1}^\infty \bigcup\limits_{j=1}^{n}\left(\bigcup\limits_{E} \{L_{j-1}(l(k),E)\}\right) = \bigcup \limits_{k=1}^{\infty} I(l(k)) = \mathbb{N}
\end{equation*}
where the second to last equality follows from Lemma 2.12 and the last equality follows from Lemma 2.14.
\end{proof}

Proposition 2.16 shows that the desired collection of sets contains every positive integer. We now need to show that no integer belongs to two of the sets simultaneously. In other words, we need to show that for any sequence $l$ satisfying the hypotheses of Theorem 2.2 and positive integers $i\not=j$
$$L_{i}(l(k_1),E_1)\not=L_{j}(l(k_2),E_2)$$
for any $k_1,k_2,E_1,E_2$. If $l(k_1) = l(k_2),$ then by Lemma 2.12, the forms must be the same, so $i = j$. On the other hand, if $l(k_1)$ and $l(k_2)$ are sufficiently far apart (note that $n$ is fixed), $L_i(l(k_1),E_1)$ and $L_j(l(k_2),E_2)$ cannot be equal. Thus a more careful analysis will be needed when $l(k_1)$ and $l(k_2)$ are unequal but close.

\begin{lemma}
$\max I(l(k)) < \min I(l(k+2))$.
\end{lemma}
\begin{proof}
We need to show that
$$l(k+2)-(2^{n-1}-1) > l(k)+(2^{n-1}-1),$$
i.e., that
$$l(k+2) - l(k) > 2^n-2.$$
Since min $G = 2^{n-1}$, we have
\begin{equation*}
    l(k+2)-l(k) = l(k+2) - l(k+1) + l(k+1) - l(k) \geq 2 \cdot 2^{n-1} = 2^n > 2^n-2.
\end{equation*}
\end{proof}

We are now ready to finish the proof of Theorem 2.2. Since increasing $k$ makes the smallest element of $I(l(k))$ larger, it is now clear that
$$I(l(k)) \cap I(l(k+b)) = \emptyset$$
for all $b \geq 2$. Thus if $L_{i}(l(k_1),E_1)$ and $L_{j}(l(k_2),E_2)$ with $i \not= j$ and $k_2 \geq k_1$ take on the same value $m$, it must be the case that $k_2 = k_1+1$ and
$$m \in I(l(k_1)) \cap I(l(k_1+1)).$$
Of course, we are done if the above intersection is empty.

We first consider those $k$ for which $l(k+1) - l(k) = \max G = 2^n - 1$. In this case,
\begin{align*}
    I(l(k)) &= [l(k)-(2^{n-1}-1), l(k)+(2^{n-1}-1)],\\
    I(l(k+1)) &= [l(k)+2^n-1-(2^{n-1}-1), l(k)+2^n-1+(2^{n-1}-1)]\\ 
              &= [l(k)+2^{n-1}, l(k) + 3(2^n)-2].
\end{align*}

Thus the smallest integer in the interval $I(l(k+1))$ is at least $l(k) + 2^{n-1}$, which is larger than the largest integer in $I(l(k))$. Hence the intersection is empty.
\\We now consider those $k$ for which $l(k+1)-l(k) = \sum_{a=1}^{b}2^{n-a}$ for $1 \leq b \leq n-1$.
\\In this case, we can write the intersection $I(l(k)) \cap I(l(k+1))$ as
\begin{align*}
    I(l(k)) \cap I(l(k+1)) &= \left[l(k) - 2^{n-1} + 1, l(k) + 2^{n-1}-1\right] \cap \left[l(k) + \sum_{a=1}^{b}2^{n-a} - 2^{n-1} + 1, \ldots\right] \\
    &= \left[l(k) + \sum_{a=1}^{b}2^{n-a} - 2^{n-1} + 1,\ l(k) + 2^{n-1}-1\right]\\
    &= \left[l(k) + 1 + \sum_{a=2}^{b}2^{n-a},\ l(k) + 2^{n-1} - 1\right].
\end{align*}
Suppose that for some $E_1,E_2$ we have

$$m = L_{i}(l(k),E_1) \text{ and } m = L_{j}(l(k+1),E_2).$$
This means that $m$ has a representation in terms of both $L_i$ and $L_j$. From the fact that $m = L_i(l(k),E_1)$, we have

$$m = l(k) + \sum _{s=2}^{i+1} \epsilon_{n-s}\cdot 2^{n-s}.$$
But since $m\in I(l(k))\cap I(l(k+1))$, we must have $m > l(k) + \sum_{a=2}^b2^{n-a}$. We claim this forces $i + 1 > b$ as well as
\[
    \epsilon_{n-2} = \epsilon_{n-3} = \ldots = \epsilon_{n-b-1} = 1.
\]
For proof of this last fact, suppose that there is some non-empty subset $I$ of the integers 2 through $b+1$ such that for $d\in I$, $\epsilon_{n-d} = -1$ while for $e\not\in I$, we have $\epsilon_{n-e} = 1$. Using the lower bound for $m$, we write
\[
    m = l(k) + \sum_{s=2}^{i+1}\epsilon_{n-s}2^{n-s} > l(k) + \sum_{a=2}^b2^{n-a}.
\]
Canceling $l(k)$ and the common terms in the sums gives
\[
    -\sum_{d\in I}2^{n-d} + \sum_{s=b+1}^{i+1}\epsilon_{n-s}2^{n-s} > \sum_{d\in I}2^{n-d} \iff \sum_{d\in I}2^{n+1-d} < \sum_{s=b+1}^{i+1}\epsilon_{n-s}2^{n-s}.
\]
Now, choosing any $a\in I$ allows us to write
\[
    2^{n+1-a} < \sum_{d\in I}2^{n-d+1} < \sum_{s=b+1}^{i+1}\epsilon_{n-s}2^{n-s} < \sum_{s=b+1}^{n}2^{n-s} = 2^{n-b} - 1 = 2^{n+1-(b+1)}-1
\]
But since $a\in I$, $a \leq b+1$. This is a contradiction.
Therefore, we have shown that the $\epsilon_{n-i}$ with $n\leq n-b-1$ are 1 and so the $L_i$ representation of $m$ must have the form
\[
    m = l(k) + \sum_{a=2}^{b+1}2^{n-a} + \sum_{s=b+2}^{i+1}\epsilon_{n-s}2^{n-s}.
\]
Since $m = L_j(l(k+1),E_2)$ we can additionally write
\[m = l(k+1) + \sum_{s=2}^{j+1}\delta_{n-s}2^{n-s} = l(k) + \sum_{a=1}^b2^{n-a} + \sum_{s=2}^{j+1}\delta_{n-s}2^{n-s}\]
because we have assumed $l(k+1) = l(k) + \sum_{a=1}^b2^{n-a}$. Next, we equate these two representations and cancel common terms, resulting in
\begin{equation*}
    \sum_{s=b+2}^{i+1}\epsilon_{n-s}2^{n-s} + 2^{n-b-1} = 2^{n-1} + \sum_{s=2}^{j+1}\delta_{n-s}2^{n-s}.
\end{equation*}
Rearranging this identity, we get
\begin{align*}
    \sum_{s=b+2}^{i+1}\epsilon_{n-s}2^{n-s} &= 2^{n-1} - 2^{n-b-1} + \sum_{s=2}^{j+1}\delta_{n-s}2^{n-s} \\
    &= \left(\sum_{c=2}^n2^{n-c} + 1\right) - \left(\sum_{c=b+2}^n2^{n-c} + 1\right) + \sum_{s=2}^{j+1}\delta_{n-s}2^{n-s} \\
    &= \sum_{c=2}^{b+1}2^{n-c} + \sum_{s=2}^{j+1}\delta_{n-s}2^{n-s}
\end{align*} 
We now wish to show that $\delta_{n-2} = \delta_{n-3} = \ldots = \delta_{n-b-1} = -1$. Suppose that there is a nonempty subset $I$ of the integers 2 through $b+1$ such that for $d\in I$, $\delta_{n-d} = 1$ while $\delta_{n-s} = -1$ for $s\not\in I$. Then, continuing from the above equality, we have
\[
    \sum_{s=b+2}^{i+1}\epsilon_{n-s}2^{n-s} = \sum_{d\in I}2^{n-d} + \sum_{s\in I}2^{n-s} + \sum_{s=b+2}^{j+1}\delta_{n-s}2^{n-s}.
\]
Rearranging, we have
\[
    \sum_{d\in I}2^{n+1-d} = \sum_{s=b+2}^{i+1}\epsilon_{n-s}2^{n-s} - \sum_{s=b+2}^{j+1}\delta_{n-s}2^{n-s}.
\]
Consider the absolute values of both sides of the inequality. The absolute values of both sums on the right side are strictly less than $2^{n-b-1}$ while the sum on the left is greater than any individual member of the sum. So, choosing any $a\in I$ gives the following inequality
\[ 
    2^{n+1-a} = 2^{n-(a-1)} <  2^{n-b-1} + 2^{n-b-1} = 2^{n-b}, 
\]
but since $a\in I$ we have $a-1 \leq b,$ which gives a contradiction. Therefore, $\delta_{n-2} = \delta_{n-3} = \ldots = \delta_{n-b-1} = -1$. Using this fact and equating the two representations of $m$, we find that
\[
    \sum_{s=b+2}^{i+1}\epsilon_{n-s}2^{n-s} = \sum_{s=b+2}^{j+1}\delta_{n-s}2^{n-s}.
\]
But $i\not=j$, so by Lemma 2.10 this is a contradiction.
Hence $L_i$ and $L_j$ have no common value, and Theorem 2.2 follows. (In particular, when we are considering $t(k) = (2^{n-1}-1)\lfloor k\phi \rfloor + k$, this forms the general $n$-columns $\phi$-partition. The properties of this partition for the $n=3$ case are analyzed in Section 4.)

\vspace{\baselineskip}

\subsection{Limiting Behavior of the n-Set Partition}

We now examine the limiting behavior of our $n$-set partitions as $n \rightarrow \infty$.

\begin{theorem}
Let $(l_k)_{k\in\mathbb{N}}$ be a sequence of sequences such that each $l_j$ satisfies the hypothesis in Theorem 2.2 with $n = j$. Let ($\mathbfcal{D}_j)_{j\in\mathbb{N}}$ be a sequence of partitions of the kind considered in Theorem 1 where $\mathbfcal{D}_1 = \{D_1\} = \{\mathbb{N}\},\ \mathbfcal{D}_2 = \{D_{1,2}, D_{2,2}\},\ \mathbfcal{D}_3 = \{D_{1,3}, D_{2,3}, D_{3,3}\}, \ldots,\mathbfcal{D}_j = \{D_{1,j},D_{2,j},\ldots,D_{j,j}\},\ldots$ where we put $\{D_{1,j}\} = \{l_j(k)\}_{k=1}^\infty$. That is, the $n$-th partition in the sequence is a partition into $n$ parts of the kind in considered in Theorem 2.2. Additionally, let $M = \{1,3,5,7,9,\ldots\}$, the set of odd positive integers. Then we have the following pointwise convergence properties of the sets within each partition as $n\rightarrow\infty$:

$$D_{n,n} \rightarrow M$$
$$D_{n-1,n} \rightarrow 2 \cdot M$$
$$\vdots$$
$$D_{n-e,n} \rightarrow 2^e \cdot M$$
$$\vdots$$

\end{theorem}

\begin{proof}

Let $n\in\mathbb{N}$. The first (smallest) element of $D_{1,n}$ is $2^{n-1}$, and it is easy to see that the first element of $D_{j,n}$ for any $1 \leq j \leq n$ is $2^{n-j}$. In particular, the first element of $D_{n,n}$ is simply 1. For further insight into the small values of the various $D_{j,n}$, we determine the range of the linear forms $L_{j}(2^{n-1},E)$, $0 \leq j \leq n-1$ as $E$ varies for any $\mathbfcal{D}_n$. For any particular $j$, $L_j(2^{n-1},E)$, can lie between a minimum of
\begin{align*}
2^{n-1}-2^{n-2}-2^{n-3}-\ldots - 2^{n-j-1}
&= 2^{n-1}-2^{n-j-1}(1+2+2^2+\ldots + 2^{j-1})\\
&= 2^{n-1}-2^{n-j-1}(2^j-1) = 2^{n-j-1}
\end{align*}
and a maximum of
\begin{align*}
2^{n-1}+2^{n-2}+2^{n-3}+\ldots + 2^{n-j-1}
&= 2^{n-1}+2^{n-j-1}(2^j-1)\\ 
&= 2^{n-j-1}(2^{j+1}-1).
\end{align*}
Dividing the value of the form by $2^{n-j-1}$ yields
\begin{equation*}
    L_j(2^{n-1},E)/2^{n-j-1} = (2^{n-1} + \sum_{s=2}^{j+1}\epsilon_{n-s}2^{n-s})/2^{n-j-1} = 2^j + \sum_{s=2}^{j+1}\epsilon_{n-s}2^{j+1-s}.
\end{equation*}
Here $L_{j}(2^{n-1},E)/2^{n-j-1}$ is odd, and can take on a total of $2^j$ values as $E$ varies. The minimum value is 1 and the maximum is $2^{j+1}-1$, a range in which there are exactly $\frac{1}{2}(2^{j+1}-1+1) = 2^j$ odd numbers. Therefore, $L_j(2^{n-1},E)/2^{n-j-1}$ takes on exactly the numbers 

\begin{equation*}
    \{1,3,5,7,\ldots,2^{j+1}-1\}
\end{equation*}
as $E$ varies over all possible values. Thus, since the first $2^{j-1}$ numbers in $D_{j,n}$ are exactly those possible values of $L_{j-1}(2^{n-1},E)$, the first $2^{j-1}$ numbers in $D_{j,n}$ are exactly
$$2^{n-j}(1,3,5,\ldots ,2^j-1).$$
Now set $j = n-e$. We see that the first $2^{n-e-1}$ numbers of each $D_{n-e,n}$ are exactly

$$2^{e}(1,3,5,\ldots , 2^{n-e}-1).$$
\end{proof}

\subsection{Consequences of Theorems 2.2 and 2.18 Regarding Beatty Sequences}

Beatty sequences and the theory surrounding them are a major motivation in establishing Theorems 2.2 and 2.18. Recall that given irrational $\alpha,\beta > 0$ which satisfy
\begin{equation*}
    \frac{1}{\alpha} + \frac{1}{\beta} = 1,
\end{equation*}
Beatty's Theorem states that the sets
\begin{equation*}
    \{\lfloor k\alpha \rfloor\}_{k=1}^\infty\quad \text{and} \quad \{\lfloor k\beta \rfloor\}_{k=1}^\infty 
\end{equation*}
partition the positive integers. Using this and Corollary 2.6, we can establish an interesting way of "extending" these partitions in the context of Theorem 1. If we let $\alpha$ be irrational such that $1<\alpha<2$, then applying Corollary 2.3 gives that the sequence
\begin{equation*}
    t(k) = \lfloor k\alpha \rfloor + k = \lfloor k(\alpha+1) \rfloor
\end{equation*}
satisfies the hypotheses of Theorem 2.2 and generates a 2-set partition. The sequence $t$ is a Beatty sequence with $2<\alpha+1<3$, so, by Beatty's Theorem, the set $D_2$ is composed of the elements of the complementary Beatty sequence $\lfloor k\frac{\alpha+1}{\alpha} \rfloor$. The given partition $\{D_1,D_2\}$ is a Beatty partition. This itself is nothing new, but Theorem 1 gives a natural way to construct a potentially new, related partition. 

For any $n>2$, we can use $\alpha$ to form
\begin{equation*}
    t(k) = (2^{n-1}-1)\lfloor k\alpha \rfloor + k
\end{equation*}
from which Corollary 2.6 gives an $n$ set partition. In addition, Theorem 2.18 tells us that as $n\rightarrow\infty$, the sequence of partitions $(\mathbfcal{D}_n)$ with $D_{1,n} = \{(2^{n-1}-1)\lfloor k\alpha \rfloor + k\}_{k=1}^\infty$ approaches the 2-adic, giving an interpolation between the 2-adic and any Beatty partition with one of $\alpha,\beta$ between 2 and 3. The rest of this paper is a study of interactions of the specific Beatty partition given by $a(k) = \lfloor k\phi \rfloor$ and $b(k) = \lfloor k\phi^2 \rfloor$ with other partitions, one of which is the $n=3$ extension of this partition given above.

To see that the Beatty partition involving $a$ and $b$ is in fact one of our partitions, notice that if we set $h(k) = \lfloor k\phi \rfloor$, we get
\begin{equation*}
    t(k) = \lfloor k\phi \rfloor + k = \lfloor k(\phi+1) \rfloor = \lfloor k\phi^2 \rfloor
\end{equation*}
since $\phi^2 = \phi + 1$. In particular, we study the $n=3$ extension in which we define $d(k) = t(k) = (2^2-1)a(k) + k = 3a(k) + k$. It is interesting to note that $\phi$ is the only positive number such that $\{h(k)\}_{k=1}^\infty = \{\lfloor k\alpha \rfloor\}_{k=1}^\infty$ is exactly the set $D_2$ when we put $D_1 = \{h(k) + k\}_{k=1}^\infty = \{\lfloor k(\alpha+1) \rfloor\}_{k=1}^\infty$. This follows from the fact that 
\begin{equation*}
    \frac{1}{\alpha} + \frac{1}{\alpha+1} = 1\quad \iff \quad \alpha^2-\alpha-1=0.
\end{equation*}

\vspace{\baselineskip}

\section{The Interaction of Two Beatty Partitions}

It is of interest to explore how the classification of integers into $A$ and $B$ numbers relates to the classifications given by other partitionings of the integers. Here $A$ and $B$ represents the specific Beatty partition described in section 2.4: $A = \{a(n)\}_{n=1}^\infty$ and $B = \{b(n)\}_{n=1}^\infty$, where $a(n) = \lfloor n\phi \rfloor$ and $b(n) = \lfloor n\phi ^2 \rfloor$. We examine this question for one of the partitions already discussed here (the case $n=3$), but begin with a simpler case that serves as a model. Also, in this simpler case we can provide a more complete description. 

\subsection{Column Classifications for Beatty Partitions}

\begin{theorem}

(The KLM formula). For integers $K, L $ and $M$ we have 
\[
    a(Ka(n) + Ln + M) = Kb(n) + La(n) + \lfloor M\phi + (L\phi - K) \frac{\{n\phi \}}{\phi } \rfloor
\]

\end{theorem}

The following important result is known, but we include a proof for the sake of completeness.

\begin{theorem}

\[
    \{a(n)\phi \} = 1 - \frac{\{ n\phi \} }{\phi },
\]
\[
    \{b(n)\phi \} = \frac{\{ n\phi \} }{\phi ^2}.
\]

\end{theorem}

For the proof we use the KLM formula and a precise form of the fact that $b(n)$ is approximately $\phi a(n)$.

\begin{lemma}

\[
a(a(n)) = a(n)+n-1 = b(n)-1,
\]
\[
a(b(n)) = a(n)+b(n).
\]

\end{lemma}

\begin{proof}

We use the KLM formula, first with $K = 1, L = M = 0$ and then with $K = L = 1, M = 0$.
\end{proof}

\begin{lemma}

$b(n) - \phi a(n) = \frac{\{ n\phi \} }{\phi}$.

\end{lemma}

\begin{proof}

Since $\phi = 1 + \frac{1}{\phi }$, the left side
\begin{align*}
    b(n) - \phi a(n) &= a(n) + n - \phi a(n) = n - \frac{a(n)}{\phi }\\
    &= \frac{n\phi - a(n)}{\phi }= \frac{\{n\phi \}}{\phi }.
\end{align*}
\end{proof}
We may now prove the above theorem.
\\Add $a(n)\phi $ to both sides of $-a(a(n)) = 1 - b(n)$ to obtain
\begin{align*}
    \{a(n)\phi \} &= a(n)\phi - \lfloor a(n)\phi \rfloor = a(n)\phi - a(a(n))\\
    &= 1 - b(n) + a(n)\phi = 1 - \frac{\{n\phi \}}{\phi}.
\end{align*}
Next, add $b(n)\phi $ to both sides of $-a(b(n)) = -a(n) - b(n)$ to obtain, since $\phi - 1 = \frac{1}{\phi }$,
\begin{align*}
    \{b(n)\phi \} &= b(n)\phi - \lfloor b(n)\phi \rfloor = b(n)\phi - a(b(n)) = b(n)\phi - b(n) - a(n)\\
    &= -a(n) + b(n) \cdot (\phi -1) = -a(n) + \frac{b(n)}{\phi } = \frac{\{n\phi \}}{\phi ^2}.
\end{align*}
This proves the theorem.

\begin{corollary}
$\{a(n)\phi \} + \phi \{b(n)\phi \} = 1$.
\end{corollary}

Beatty's Theorem applied to $\frac{1}{\phi ^2 /2} + \frac{1}{\phi ^3} = 1$ shows that the infinite sequences
\[
c(n) = \lfloor n \cdot \frac{\phi ^2}{2} \rfloor \text{ and } d(n) = \lfloor n\phi ^3 \rfloor , n = 1, 2, 3, \ldots
\]
partition the set of positive integers. Call these two sequences (we may also call them sets) $C$ and $D$. We shall investigate how the elements of $C$ and $D$ are distributed among the sets $A$ and $B$, and vice-versa. The following result illuminates this subject.

\begin{theorem}
For any $n \geq 1$ the fractional part $\{d(n)\phi \}$ is in the interval 
$(\frac{3-\sqrt{5}}{2}, \frac{1}{2}) \text{ if } \{n\phi \} > \frac{1}{2}, \text{ and in the interval }
(\frac{4-\sqrt{5}}{2}, 1) \text{ if } \{n\phi \} < \frac{1}{2}$.
Each interval clearly has length $\frac{\sqrt{5}-2}{2}$.
\\For any $n \geq 1$ the fractional part $\{c(2n)\phi \}$ lies in the interval
$(0, \frac{3-\sqrt{5}}{2})$ while the fractional part $\{c(2n+1)\phi \}, n \geq 1, $ lies in the interval $(\frac{1}{2}, \frac{4-\sqrt{5}}{2})$.
Each of these last two intervals clearly has length $\frac{3-\sqrt{5}}{2}$.
\end{theorem}

\begin{remark}
These intervals are all disjoint and their total length is 1. It is curious how the criteria for interval membership is defined by an inequality for $d(n)$ and by parity for $c(n)$.
\end{remark}

We first establish a lemma.

\begin{lemma}
\[
d(n) =\left\{ \begin{array}{ll}
	   2 \cdot a(n) + n,  &  \text{ if } \{n\phi \} < \frac{1}{2},\\
	   2 \cdot a(n) + n+1, & \text{ if } \{n\phi\} > \frac{1}{2}.
	\end{array}\right.
\]
\end{lemma}

\begin{proof}
By the KLM formula,
\begin{align*}
    d(n) &= \lfloor n + 2\phi n \rfloor = n + a(2n) = n + 2a(n) + \lfloor (2\phi ) \cdot \frac{\{n\phi \}}{\phi } \rfloor \\
    &= n + 2a(n) + \lfloor 2 \{n\phi \} \rfloor
\end{align*}
where we set $K = M = 0$ and $L = 2$. The result follows.
\end{proof}
We can now prove the assertions about $d(n)$. First, say $\{n\phi \} < \frac{1}{2}$. Then
\begin{align*}
    \{d(n) \phi \} &= \{n\phi + 2 \cdot a(n)\phi \} = \{\{n\phi \} + 2\cdot \{a(n)\phi \} \}\\
    &= \{\{n\phi \} + 2\cdot (1-\frac{\{n\phi \}}{\phi }) \} = \{(2-\sqrt{5})\{n\phi \} \}\\
    &= 1 + (2-\sqrt{5})\{n\phi \},
\end{align*}
so $\{d(n)\phi \}$ lies in $(\frac{4-\sqrt{5}}{2}, 1)$. Next, say $\{n\phi \} > \frac{1}{2}$. Then
\begin{align*}
    \{d(n) \phi \} &= \{n\phi + \phi -1 + 2 \cdot a(n)\phi \} = \{\{n\phi \} + 2\cdot (1-\frac{\{n\phi \}}{\phi })+\frac{1}{\phi } \}\\
    &= \{(2-\sqrt{5})\{n\phi \} + \frac{1}{\phi } \}.
\end{align*}
This has the form $\{p\}$ where $0 < p < 1$. Thus $\{d(n)\phi \} = \frac{1}{\phi } - (\sqrt{5}-2)\{n\phi \}$ and thus $\{d(n)\phi \}$ lies in $(\frac{3-\sqrt{5}}{2}, \frac{1}{2})$.
\vspace{\baselineskip}
\\The case of $c(2n)\phi $ is straightforward.
\[
\{c(2n)\phi \} = \{\lfloor (2n)\cdot \frac{\phi ^2}{2} \rfloor \phi \} = \{b(n)\phi \} = \frac{\{n\phi \}}{\phi ^2} \in (0, \frac{3-\sqrt{5}}{2}).
\]
A more detailed consideration is needed for $\{c(2n+1)\phi \}$. Set $\lambda = \frac{5-\sqrt{5}}{4}$. Since $\frac{\phi ^2}{2} = \frac{3+\sqrt{5}}{4}$ we have $\lambda < 1$ and $\frac{\phi ^2}{2} + \lambda = 2$.

\begin{lemma}
\[
\phi ^3\{c(2n+1)\phi \}-\phi \{n\phi \} =\left\{ \begin{array}{ll}
	   \phi ^2,  &  \text{ if } \{n\phi \} < \lambda,\\
	   1, & \text{ if } \{n\phi\} > \lambda.
	\end{array}\right.
\]
\end{lemma}

\begin{proof}
Let
\[
e(n) =\left\{ \begin{array}{ll}
	   1,  &  \text{ if } \{n\phi \} < \lambda,\\
	   2, & \text{ if } \{n\phi\} > \lambda.
	\end{array}\right.
\]
Since $0 < \{n\phi \} < 1 < \frac{\phi ^2}{2} < 2 = \frac{\phi ^2}{2} + \lambda, \text{ and } 1+\frac{\phi ^2}{2} < 3$, we have that $\{n\phi \} + \frac{\phi ^2}{2}$ is strictly between 1 and 3. So the above equality $2 = \frac{\phi ^2}{2} + \lambda$ shows that $\lfloor \{n\phi \} + \frac{\phi ^2}{2} \rfloor = e(n)$.
Now
\begin{align*}
    c(2n+1) &= \lfloor (2n+1) \cdot \frac{\phi ^2}{2} \rfloor = \lfloor n\phi ^2 + \frac{\phi ^2}{2} \rfloor \\
    &= \lfloor b(n) + \{n\phi ^2\} + \frac{\phi ^2}{2} \rfloor = \lfloor b(n) + \{n\phi \} + \frac{\phi ^2}{2} \rfloor \\
    & = b(n) + \lfloor \{n\phi \} + \frac{\phi ^2}{2} \rfloor = b(n) + e(n).
\end{align*}

In the case $\{n\phi \} < \lambda \text{ }(e(n) = 1)$ we have $\{c(2n+1)\phi \} = \{b(n)\phi + \phi \} = \{\{b(n)\phi \} + \phi \} = \{\frac{\{n\phi \} }{\phi ^2} + \frac{1}{\phi }\} = \frac{\{n\phi \} }{\phi ^2} + \frac{1}{\phi }$ since $\frac{1}{\phi } + \frac{1}{\phi ^2} = 1$. Thus in this case $\phi ^3\{c(2n+1)\phi \} - \phi \{n\phi \} = \phi ^2$.

In the case $\{n\phi \} > \lambda \text{ }(e(n) = 2)$ we have $\{c(2n+1)\phi \} = \{b(n)\phi + 2\phi \} = \{\frac{\{n\phi \} }{\phi ^2} + \frac{2}{\phi }\}$.
Now $1 < \frac{2}{\phi } < \frac{\{n\phi \} }{\phi ^2} + \frac{2}{\phi } < \frac{1}{\phi ^2} + \frac{2}{\phi } = \phi < 2$, so $\{c(2n+1)\phi \} = \frac{\{n\phi \} }{\phi ^2} + \frac{2}{\phi }-1$. This implies that $\phi ^3\{c(2n+1)\phi \} = \phi \{n\phi \} + 2\phi ^2 - \phi ^3$, so $\phi ^3\{c(2n+1)\phi \} - \phi \{n\phi \} = 2\phi ^2 - \phi ^3 = 1$ and the result follows.
\end{proof}

Now observe that $\frac{\lambda }{\phi ^2}+\frac{1}{\phi ^3} = \frac{1}{2} < \frac{1}{\phi } < \frac{\lambda }{\phi ^2} + \frac{1}{\phi } = \frac{4-\sqrt{5}}{2}$.

In the $\phi ^2$ case ($\{n\phi \}< \lambda $) this shows that $\frac{1}{2} < \frac{1}{\phi } < \{c(2n+1)\phi \} < \frac{4-\sqrt{5}}{2}$ where in particular the upper bound is achieved. In the 1 case ($\{n\phi \}> \lambda $) this shows that $\frac{1}{2} < \{c(2n+1)\phi \} < \frac{1}{\phi } < \frac{4-\sqrt{5}}{2}$ where in particular the lower bound is achieved.

This proves the theorem.
\vspace{\baselineskip}

Summary of Fractional Parts Identities:
\[
\phi \{a(n)\phi \} + \{n\phi \} = \phi
\]
\[
\phi ^2\{b(n)\phi \} - \{n\phi \} = 0
\]
\[
\{d(n)\phi \} + (\sqrt{5}-2)\{n\phi \}=\left\{ \begin{array}{ll}
	   1,  &  \text{ if } \{n\phi \} < \frac{1}{2},\\
	   \frac{1}{\phi }, & \text{ if } \{n\phi\} > \frac{1}{2}
	\end{array}\right.
\]
\[
\phi ^3\{c(m)\phi \} - \phi \{n\phi \}=\left\{ \begin{array}{ll}
	   0,  &  \text{ if } m = 2n,\\
	   \phi ^2, & \text{ if } m = 2n+1, \{n\phi\} < \lambda,\\
	   1, & \text{ if } m = 2n+1, \{n\phi\} > \lambda.
	\end{array}\right.
\]

We may now use Weyl's Theorem on uniform distribution which asserts that for $\alpha$ irrational the sequence $\{n\alpha \}_{n=1}^{\infty}$ is uniformly distributed in (0, 1).

Divide the unit interval (0, 1) into 4 parts ($I_1, I_2, I_3, I_4$) that are respectively 
$$(0, \frac{1}{\phi ^2}), (\frac{1}{\phi ^2}, \frac{1}{2}), (\frac{1}{2}, \frac{4-\sqrt{5}}{2}), (\frac{4-\sqrt{5}}{2}, 1).$$

Our results imply that
\begin{align*}
    \{n\phi \} &\in I_1 \Longleftrightarrow n \in B\\
    \{n\phi \} &\in I_2 \cup I_3 \cup I_4 \Longleftrightarrow n \in A\\
    \{n\phi \} &\in I_1 \cup I_3 \Longleftrightarrow n \in C\\
    \{n\phi \} &\in I_2 \cup I_4 \Longleftrightarrow n \in D.
\end{align*}

From this we see that
\begin{align*}
(c(n), d(n)) &\in (A \times A) \cup (B \times A),\\
(a(n), b(n)) &\in (C \times C) \cup (D \times C).
\end{align*}

Note that for arbitrary two integer vectors $(h, k)$, a partitioning of the integers into 2 sequences gives 4 membership classifications for $(h, k)$, so it is notable here that for the $(A, B)$ and $(C, D)$ partitionings we have for the Wythoff pairs $(a(n), b(n))$ and the $(c(n), d(n))$ pairs only 2 possibilities.

We can also be a bit more quantitative. Since $\frac{4-\sqrt{5}}{2} - \frac{1}{2} = \frac{1}{\phi ^2}$, we see that the set of integers $n$ for which $c(n) \in A$ has density $\frac{1}{2}$ and the same for $c(n) \in B$. Since $1-\frac{1}{\phi ^2} = \frac{1}{\phi }$ we see that the set of integers $n$ for which $a(n) \in C$ has density $(\frac{1}{\phi ^2})/(\frac{1}{\phi }) = \frac{1}{\phi}$. Hence the set of integers $n$ for which $a(n) \in D$ has density $\frac{1}{\phi ^2}$.

\subsection{Extension of Fractional Parts Identities}

In connection with the fractional parts identities, we also have the following result related to Fibonacci numbers which can be used to solve more interaction problems.

\begin{theorem}
Define $F_{k}$ as the k-th Fibonacci number for any $k \in \mathbb{N}$. Let $r$ be an odd positive integer and let $m$ and $n$ be any positive integers. Then we have $\phi ^{r} \{m\phi \} - \phi ^{r-2} \{n\phi \} = 1$ if and only if $m = a(n) + n + F_{r}$.
\end{theorem}

Before we prove the theorem, we need to establish the following Lemmas.

\begin{lemma}
When $r$ is an odd positive integer, $\lfloor F_{r}\phi + (\phi -1)\cdot \frac{\{n\phi \}}{\phi } \rfloor = F_{r+1}$.
\end{lemma}

\begin{proof}
First we will prove that $F_{r}\phi + (\phi -1)\cdot \frac{\{n\phi \}}{\phi } > F_{r+1}$. Note that Fibonacci numbers have a closed-form solution (often known as Binet's formula): $F_{k} = \frac{1}{\sqrt{5}}(\phi ^{k}-(-1)^{k}\cdot (\frac{1}{\phi })^{k})$ for all $k \in \mathbb{N}$. Since $F_{r}\phi + (\phi -1)\cdot \frac{\{n\phi \}}{\phi } > F_{r}\phi + 0 = F_{r}\phi$, and
\[
F_{r}\phi > F_{r+1} \Longleftrightarrow \frac{1}{\sqrt{5}}(\phi ^{r+1} + \frac{1}{\phi ^{r-1}}) > \frac{1}{\sqrt{5}}(\phi ^{r+1} - \frac{1}{\phi ^{r+1}}),
\]
which must be true because $\frac{1}{\phi ^{k}} > 0$ for all $k \in \mathbb{N}$, we have $F_{r}\phi + (\phi -1)\cdot \frac{\{n\phi \}}{\phi } > F_{r+1}$.

Next, we will prove that $F_{r}\phi + (\phi -1)\cdot \frac{\{n\phi \}}{\phi } < F_{r+1} + 1$. Since $F_{r}\phi + (\phi -1)\cdot \frac{\{n\phi \}}{\phi } < F_{r}\phi + (\phi -1)\cdot \frac{1}{\phi }$, and
\[
F_{r}\phi + (\phi -1)\cdot \frac{1}{\phi } \leq F_{r+1}+1 \Longleftrightarrow F_{r}\phi - \frac{1}{\phi} \leq F_{r+1} \Longleftrightarrow \frac{1}{\sqrt{5}}(\frac{1}{\phi ^{r-1}} + \frac{1}{\phi ^{r+1}}) \leq \frac{1}{\phi },
\]
which must be true because $\frac{1}{\sqrt{5}}(\frac{1}{\phi ^{r-1}} + \frac{1}{\phi ^{r+1}}) \leq \frac{1}{\sqrt{5}}(1+\frac{1}{\phi ^2}) = \frac{1}{\phi }$, we have $F_{r}\phi + (\phi -1)\cdot \frac{\{n\phi \}}{\phi } < F_{r+1} + 1$.

Since $F_{r}\phi + (\phi -1)\cdot \frac{\{n\phi \}}{\phi } \in (F_{r+1}, F_{r+1}+1)$, it is obvious that $\lfloor F_{r}\phi + (\phi -1)\cdot \frac{\{n\phi \}}{\phi } \rfloor = F_{r+1}$.
\end{proof}

\begin{lemma}
$\phi ^{k} = F_{k}\phi + F_{k-1}$ for all $k \in \mathbb{N}$.
\end{lemma}

\begin{proof}
Fibonacci numbers are recursively defined such that $F_{k+2} = F_{k+1} + F_{k}$ for all $k \in \mathbb{N}$, and $F_0 = 0, F_1 = 1, F_2 = 1, F_3 = 2, F_4 = 3, F_5 = 5, \ldots$. So when $k = 1$, $\phi ^{k} = \phi = F_{1}\phi + F_{0} = F_{k}\phi + F_{k-1}$.

Suppose that $\phi ^{k} = F_{k}\phi + F_{k-1}$ for some $k \in \mathbb{N}$. Then
\begin{align*}
    \phi ^{k+1} &= \phi ^{k} \cdot \phi = (F_{k}\phi + F_{k-1}) \phi\\
    &= F_{k}\phi ^2 + F_{k-1}\phi = F_{k}(\phi +1) + F_{k-1}\phi\\
    &= (F_{k} + F_{k-1})\phi + F_{k} = F_{k+1}\phi + F_{k}.
\end{align*}
So by induction, we have $\phi ^{k} = F_{k}\phi + F_{k-1}$ for all $k \in \mathbb{N}$.
\end{proof}

Now, we may proceed to the proof of Theorem 3.10. Suppose that $m = a(n) + n + F_{r}$ for some odd positive integer $r$. Using the KLM formula with $K = 1, L = 1, M = F_{r}$, we get
\[
\{m\phi \} = (a(n)+n+F_{r})\phi - (b(n)+a(n)+\lfloor F_{r}\phi + (\phi -1)\cdot \frac{\{n\phi \}}{\phi } \rfloor )
\]
Now, since $\lfloor F_{r}\phi + (\phi -1)\cdot \frac{\{n\phi \}}{\phi } \rfloor = F_{r+1}$ by Lemma 3.11, we have
\begin{align*}
    \{m\phi \} &= \lfloor n\phi \rfloor \phi - \lfloor n\phi ^2 \rfloor + \{n\phi \} + F_{r}\phi - F_{r+1} \\
    &= (n\phi - \{n\phi \})\phi - \lfloor n\phi \rfloor - n + \{n\phi \} + F_{r}\phi - F_{r+1}\\
    &= (2 - \phi )\{n\phi \} + F_{r}\phi - F_{r+1},
\end{align*}
which means
\begin{align*}
    \phi ^{r}\{m\phi \} &= (2-\phi )\phi ^{r}\{n\phi \} + F_{r}\phi ^{r+1} - F_{r+1}\phi ^{r}\\
    &= (1-1/\phi )\phi ^{r}\{n\phi \} + F_{r}\phi ^{r+1} - F_{r+1}\phi ^{r}\\
    &= \phi ^{r-1}(\phi -1)\{n\phi \} + F_{r}\phi ^{r+1} - F_{r+1}\phi ^{r}\\
    &= \phi ^{r-2}\{n\phi \} + F_{r}\phi ^{r+1} - F_{r+1}\phi ^{r}.
\end{align*}
From Lemma 3.12, we know that $\phi ^{r+1} = F_{r+1}\phi + F_{r}$ and $\phi ^{r} = F_{r}\phi + F_{r-1}$. So $F_{r}\phi ^{r+1} - F_{r+1}\phi ^{r} = F_{r} ^2 - F_{r+1} \cdot F_{r-1}$. Cassini's identity (see [p. 41] of \cite{Cassini} for proof) states that 
$$(-1)^{n} = F_{n+1} \cdot F_{n-1} - F_{n} ^2$$ 
for $n \in \mathbb{N}$. Using this and the fact that $r$ is an odd positive integer, we have $F_{r}\phi ^{r+1} - F_{r+1}\phi ^{r} = 1$. Therefore, when $m = a(n) + n + F_{r}$ for some odd positive integer $r$, we have $\phi ^{r} \{m\phi \} - \phi ^{r-2} \{n\phi \} = 1$.
\vspace{\baselineskip}

On the other hand, if $\phi ^{r} \{m\phi \} - \phi ^{r-2} \{n\phi \} = 1$ for some odd positive integer $r$, then we can assume that $m = Ka(n)+Ln+M$, where $K, L, M \in \mathbb{Z}$, and $M$ is some constant independent of $n$. Using the KLM formula, we have
\begin{align*}
    \{m\phi \} &= m\phi - \lfloor m\phi \rfloor = Ka(n)\phi + Ln\phi + M\phi - Kb(n)-La(n)- \lfloor M\phi + (L\phi - K)\cdot \frac{\{n\phi \}}{\phi } \rfloor \\
    &= L\{n\phi \} + Ka(n)(\phi -1) - Kn + M\phi - \lfloor M\phi + (L\phi - K)\cdot \frac{\{n\phi \}}{\phi } \rfloor \\
    &= L\{n\phi \} + K(n\phi - \{n\phi \})(\phi -1) - Kn + M\phi - \lfloor M\phi + (L\phi - K)\cdot \frac{\{n\phi \}}{\phi } \rfloor \\
    &= L\{n\phi \} - K\{n\phi \}\cdot (\phi -1) + M\phi - \lfloor M\phi + (L\phi - K)\cdot \frac{\{n\phi \}}{\phi } \rfloor \\
    &= \{n\phi \}\cdot (L - K(\phi -1)) + M\phi - \lfloor M\phi + (L\phi - K)\cdot \frac{\{n\phi \}}{\phi } \rfloor.
\end{align*}
From Lemma 3.12, we know that $\phi ^{r+1} = F_{r+1}\phi + F_{r}$ and $\phi ^{r} = F_{r}\phi + F_{r-1}$. So,
\begin{align*}
    \phi ^{r}\{m\phi \} &= \{n\phi \}\cdot (L - K(\phi -1))\phi ^{r} + M\phi ^{r+1} - \lfloor M\phi + (L\phi - K)\cdot \frac{\{n\phi \}}{\phi } \rfloor \cdot \phi ^{r}\\
    &= \{n\phi \} (L - K(\phi -1))\phi ^{r} + M (F_{r+1}\phi + F_{r}) - \lfloor M\phi + (L\phi - K) \frac{\{n\phi \}}{\phi } \rfloor \cdot (F_{r}\phi + F_{r-1}).
\end{align*}
We also have $\phi ^{r} \{m\phi \} = \phi ^{r-2} \{n\phi \} + 1$ by assumption. So, as $M\cdot F_{r} - \lfloor M\phi + (L\phi - K)\cdot \frac{\{n\phi \}}{\phi } \rfloor F_{r-1}$ is an integer and $\{n\phi \} (L - K(\phi -1))\phi ^{r} + M\cdot F_{r+1}\phi - \lfloor M\phi + (L\phi - K) \frac{\{n\phi \}}{\phi } \rfloor \cdot F_{r}\phi $ is irrational, we have the following two equations:
$$ \left\{ \begin{array}{ll}
\phi ^{r-2}\{n\phi \} = \{n\phi \} (L - K(\phi -1))\phi ^{r} + M\cdot F_{r+1}\phi - \lfloor M\phi + (L\phi - K) \frac{\{n\phi \}}{\phi } \rfloor \cdot F_{r}\phi ,\\
1 = M\cdot F_{r} - \lfloor M\phi + (L\phi - K)\cdot \frac{\{n\phi \}}{\phi } \rfloor \cdot F_{r-1}.
\end{array}\right.$$
From the first equation above, we have
\[
\phi ^{r-3}\{n\phi \} = \{n\phi \} (L - K(\phi -1))\phi ^{r-1} + M\cdot F_{r+1} - \lfloor M\phi + (L\phi - K) \frac{\{n\phi \}}{\phi } \rfloor \cdot F_{r}.
\]
Since $M\cdot F_{r+1} - \lfloor M\phi + (L\phi - K) \frac{\{n\phi \}}{\phi } \rfloor \cdot F_{r}$ is an integer and $\{n\phi \} (L - K(\phi -1))\phi ^{r-1}$ is irrational, we have the following two equations:
$$ \left\{ \begin{array}{ll}
	   0 = M\cdot F_{r+1} - \lfloor M\phi + (L\phi - K) \frac{\{n\phi \}}{\phi } \rfloor \cdot F_{r},\\
	    \phi ^{r-3} = (L - K(\phi -1))\phi ^{r-1}.
\end{array}\right.$$
From the second equation above, we get
\[
L - K(\phi -1) = L - \frac{K}{\phi } = \frac{1}{\phi ^2} \Longleftrightarrow L\phi ^2 - K\phi = 1 \Longleftrightarrow (L - K)\phi + L = 1 \Longleftrightarrow L = K = 1.
\]
With $L = K = 1$, we now have
$$ \left\{ \begin{array}{ll}
	   1 = M\cdot F_{r} - \lfloor M\phi + (\phi - 1)\cdot \frac{\{n\phi \}}{\phi } \rfloor \cdot F_{r-1},\\
	    0 = M\cdot F_{r+1} - \lfloor M\phi + (\phi - 1)\cdot \frac{\{n\phi \}}{\phi } \rfloor \cdot F_{r}.
	\end{array}\right.$$
By multiplying $F_{r}$ on both sides of the first equation and multiplying $F_{r-1}$ on both sides of the second equation, this is equivalent to the following:
$$ \left\{ \begin{array}{ll}
	   F_{r} = M\cdot F_{r}^2 - \lfloor M\phi + (\phi - 1)\cdot \frac{\{n\phi \}}{\phi } \rfloor \cdot F_{r-1}F_{r},\\
	    0 = M\cdot F_{r+1}F_{r-1} - \lfloor M\phi + (\phi - 1)\cdot \frac{\{n\phi \}}{\phi } \rfloor \cdot F_{r}F_{r-1}.
	\end{array}\right.$$
Subtracting the two equations gives $M(F_{r}^2 - F_{r+1}\cdot F_{r-1}) = F_{r}$. By Cassini's identity and the fact that $r$ is odd, $F_{r}^2 - F_{r+1}\cdot F_{r-1} = 1$. So $M = F_{r}$. By Lemma 3.11, $\lfloor F_{r}\phi + (\phi -1)\cdot \frac{\{n\phi \}}{\phi } \rfloor = F_{r+1}$, the above equations are satisfied.

Therefore, when $r$ is an odd positive integer, $\phi ^{r} \{m\phi \} - \phi ^{r-2} \{n\phi \} = 1$ if and only if $m = a(n) + n + F_{r}$, where $F_{r}$ is the $r$-th Fibonacci number.

\begin{remark}
In section 3.1, we presented the following results in the summary of fractional parts identities: $\phi ^{3}\{c(k)\phi \} - \phi \{n\phi \} = \phi ^2$ if $k = 2n+1, \{n\phi \} < \lambda$, and $\phi ^{3}\{c(k)\phi \} - \phi \{n\phi \} = 1$ if $k = 2n+1, \{n\phi \} > \lambda$. Here $\lambda = \frac{5-\sqrt{5}}{4}$. The first identity is equivalent to $\phi \{c(k)\phi \} - \phi ^{-1}\{n\phi \} = 1 $ if $k = 2n+1, \{n\phi \} < \lambda$. 

Note that these are the special cases of the result in Theorem 3.10, where $r = 1, 3$ respectively. In these cases, $m = c(k)$ can be written as $c(2n+1) = \lfloor (2n+1)\cdot \frac{\phi ^2}{2} \rfloor = \lfloor n(\phi +1) + \frac{\phi +1}{2} \rfloor = a(n) + n + \lfloor \{n\phi \} + \frac{\phi +1}{2} \rfloor$, which equals $a(n)+n+1$ $(F_1 = 1)$ when $\{n\phi \} < \lambda$ and equals $a(n)+n+2$ $(F_3 = 2)$ when $\{n\phi \} > \lambda$.
\end{remark}

\section{The Interaction of the A,B Beatty Partition and its 3-Set Extension}

Now, we take a look at the interaction between the partition given by the sequences $a(k) = \lfloor k\phi \rfloor$ and $b(k) = \lfloor k\phi ^{2} \rfloor$ and the 3-set partition where $l(k) = t(k) = (2^{n-1}-1)h(k) + k$ and $h(k) = \lfloor k\phi \rfloor $. It will be useful from this point on to consider the sets of these partitions as columns. For any $n$-set partition, we will call $D_1$ the first column, $D_2$ the second column, and so on, paralleling the presentation in Table 1. We first prove some important results of 3-set partition which will turn out to be useful in the computation for the densities of classified columns later on.

\subsection{Finding the Second Column of the n-Column Extension}

Any partition of the type described in Theorem 2.2 is determined by the sequence $l$ used to define the set $D_1$. In the $n$-column extension of the $A,B$ partition, we use Corollary 2.6 to construct the following sequence
\begin{equation*}
    t(k) = (2^{n-1}-1)a(k) + k,
\end{equation*}
and we set $l(k) = t(k)$ so that $D_1 = \{t(k)\}_{k=1}^\infty$. It turns out that, in this case, we can find a simple closed form for the sequence which gives the elements of the set $D_2$. We develop this in the next theorem. First, we need a simple lemma.

\begin{lemma}
$A \cup\ (A+1) = \mathbb{N}$
\end{lemma}

\begin{proof}
This follows from the fact that $a(1) = 1$ and $1 \leq a(k+1) - a(k) \leq 2$ for all $k$.
\end{proof}

\begin{theorem}
    Let $D_1,D_2,\ldots,D_n$ be the $n$-set partition generated by $D_1(k) = (2^{n-1}-1)a(k) + k$. If we form a sequence $D_2(k)$ by placing the elements of $D_2$ in increasing order $d_{2,1} < d_{2,2} < d_{2,3} \ldots$ and defining $D_2(k) = d_{2,k}$, then
    \begin{equation*}
        D_2(k) = a(k) + (2^{n-1}-2)k - (2^{n-2} - 1).
    \end{equation*}
    
\end{theorem}

\begin{proof}
First, we will show that
\begin{align*}
    \{D_2(a(l))\} &= \{D_1(k) - 2^{n-2}\},\\
    \{D_2(a(l)+1)\} &= \{D_1(k) + 2^{n-2}\}.
\end{align*}
This will be of use because it will be easier to define $D_2$ on each of these domains separately and combine them. Since the union of the domains is all of the positive integers by Lemma 4.1, we will have the formula for all positive integers. When $l = 1$, we have
\[
D_2(a(l)) = D_2(a(1)) = D_2(1) = D_1(1)-2^{n-2} = D_1(l)-2^{n-2},
\]
\[
D_2(a(l)+1) = D_2(a(1)+1) = D_2(2) = D_1(1)+2^{n-2} = D_1(l)+2^{n-2}.
\]
We also have
\begin{align*}
    D_2(a(l+1)) = D_2(a(l)+1) &\iff a(l+1) = a(l) + 1\\
    &\iff (2^{n-1}-1)a(l+1) + (l+1) - 2^{n-2} = (2^{n-1}-1)a(l) + l + 2^{n-2} \\
    &\iff D_1(l+1) - 2^{n-2} = D_1(l) + 2^{n-2}.
\end{align*}
I.e., If $a(l)+1$ is the next $A$ number, $a(l+1)$, then $D_1(l) + 2^{n-2} = D_2(a(l)+1) = D_2(a(l+1)) = D_1(l+1) - 2^{n-2}$. If $a(l)+1\not=a(l+1)$, then $a(l)+2 = a(l+1)$ and we have
\begin{align*}
    D_2(a(l+1)) &= D_2(a(l)+2)\\
    &= (2^{n-1}-1)a(l+1)+(l+1)-2^{n-2} \text{ } = D_1(l+1)-2^{n-2}\\
    &\neq (2^{n-1}-1)a(l)+l+2^{n-2} \text{ } = D_1(l)+2^{n-2}.
\end{align*}
This implies
\begin{align*}
    D_2(a(l))&=D_1(l)-2^{n-2}=(2^{n-1}-1)a(l)+l-2^{n-2} \text{ and }\\
    D_2(a(l)+1)&=D_1(l)+2^{n-2}=(2^{n-1}-1)a(l)+l+2^{n-2}
\end{align*}
for all $l \in \mathbb{N}$.
\\In the $a(l)$ case, $D_2(a(l))=D_1(l)-2^{n-2}=(2^{n-1}-1)a(l)+l-2^{n-2}$. On the other hand, we have
\[
a(a(l))+(2^{n-1}-2)a(l)-(2^{n-2}-1)
\]
\[
= (2^{n-1}-1)a(l)-2^{n-2}+a(a(l))-a(l)+1.
\]
From KLM formula ($K = 1, L = 0, M = 0$), we have $a(a(l)) = b(l)-1 = a(l)+l-1$. So 
\[
D_2(a(l)) = a(a(l))+(2^{n-1}-2)a(l)-(2^{n-2}-1).
\]
In the $a(l)+1$ case, $D_2(a(l)+1)=D_1(l)+2^{n-2}=(2^{n-1}-1)a(l)+l+2^{n-2}$. On the other hand, we have
\begin{align*}
    &a(a(l)+1)+(2^{n-1}-2)(a(l)+1)-(2^{n-2}-1)\\
    &= (2^{n-1}-2)a(l)+a(a(l)+1)+2^{n-2}-1\\
    &= (2^{n-1}-1)a(l)+2^{n-2}+a(a(l)+1)-a(l)-1.
\end{align*}
From KLM formula ($K = 1, L = 0, M = 1$), we have $a(a(l)+1) = b(l)+1 = a(l)+l+1$. 
\\So 
\[
D_2(a(l)+1) = a(a(l)+1)+(2^{n-1}-2)(a(l)+1)-(2^{n-2}-1).
\]
So, since $A \cup (A+1) = \mathbb{N}$, the formula is valid for all $k$.
\end{proof}

\subsection{Fractional Parts Identities with respect to the 3-set Partition}

Let $D$ be column $D_1$, $C$ be column $D_2$, and $S$ be column $D_3$ in the definition of 3-set partition.

From the definition of $h(k)$ and $t(k)$, we already know that $D = \{d(k)\} = \{(2^{n-1}-1)a(k) + k\}$ and that $C = \{c(k)\} = \{d(k)-2^{n-2}\} \cup \{d(k)+2^{n-2}\}$. From Theorem 4.2, we get the formula $c(k) = a(k) + (2^{n-1}-2)k - (2^{n-2}-1)$ for column $C$. So in the 3-set case, we have
\begin{align*}
    D &= \{d(k)\} = \{3a(k)+k\},\\
    C &= \{c(k)\} = \{a(k)+2k-1\},\\
    S &= \{s(k)\} = \{c(k)-1\} \cup \{c(k)+1\}.
\end{align*}
Now we will derive some fractional parts identities with respect to $d(k), c(k),$ and $s(k)$, which will be of use in the investigation of the interaction between the $A, B$ Beatty partition and the 3-set partition.
\begin{theorem}
$\{d(k)\phi\} = 1- \frac{\sqrt{5}}{\phi^{2}}\{k\phi\}$.
\end{theorem}

\begin{proof}

Using the identity $\phi ^2 = \phi + 1$ and the KLM formula with $K = 3, L = 1, M = 0$, we can write
\begin{align*}
\{d(k)\phi\} &= d(k)\phi - \lfloor d(k)\phi \rfloor = (3a(k)+k)\phi - a(3a(k)+k)\\
&=3a(k)\phi + k\phi -3b(k)-a(k) - \lfloor (\phi -3)\cdot \frac{\{k\phi \}}{\phi } \rfloor.
\end{align*}
Now, since $\frac{\phi -3}{\phi } < (\phi -3)\cdot \frac{\{k\phi \}}{\phi} < 0$, we have that $\lfloor (\phi -3)\cdot \frac{\{k\phi \}}{\phi } \rfloor = -1$. This gives
\begin{align*}
\{d(k)\phi)\} &= 3(k\phi -\{k\phi \})\phi + \{k\phi \} -3(k\phi -\{ k\phi \})-3k+1\\
&= \{k\phi \}(4-3\phi )+1 = 1-\frac{\sqrt{5}}{\phi ^2}\{k\phi \}.
\end{align*}
\end{proof}

\begin{theorem}
$$\{c(k)\phi\} =\left\{ \begin{array}{ll}
	   \frac{\sqrt{5}}{\phi}\{k\phi\}+\frac{1-\sqrt{5}}{2},  &  \text{ if } \{k\phi\} > \frac{1}{\sqrt{5}},\\
	    \frac{\sqrt{5}}{\phi}\{k\phi\}+\frac{3-\sqrt{5}}{2}, & \text{ if } \{k\phi\} < \frac{1}{\sqrt{5}}.
	\end{array}\right.$$
\end{theorem}

\begin{proof}

Using the KLM formula where $K = 1, L = 2, M = -1$,
\begin{align*}
    \{c(k)\phi\} &= c(k)\phi - \lfloor c(k)\phi \rfloor = (a(k)+2k-1)\phi -a(a(k)+2k-1)\\
    &= a(k)\phi + 2k\phi - \phi - b(k) -2a(k) - \lfloor -\phi + (2\phi -1)\cdot \frac{\{k\phi \}}{\phi } \rfloor
\end{align*}
Since $-\phi < -\phi + (2\phi -1)\cdot \frac{\{k\phi \}}{\phi } < 2-\phi - \frac{1}{\phi } = 3 - 2\phi $, where $-\phi = -1.6180339\ldots$ and $3-2\phi = -0.2360679\ldots $, we need to consider two cases: 
$-\phi + (2\phi -1)\cdot \frac{\{k\phi \}}{\phi } \in (-\phi , -1)$ and $-\phi + (2\phi -1)\cdot \frac{\{k\phi \}}{\phi } \in (-1, 3-2\phi )$, which are equivalent to $\{k\phi \} \in (0, \frac{1}{\sqrt{5}})$ and $\{k\phi \} \in (\frac{1}{\sqrt{5}}, 1)$ respectively.
\\In the first case, we have $ \lfloor -\phi + (2\phi -1)\cdot \frac{\{k\phi \}}{\phi } \rfloor = -2$, so
\begin{align*}
    \{c(k)\phi \} &= (k\phi - \{k\phi \})\phi + 2k\phi - \phi -(k\phi - \{k\phi \}) - k - 2(k\phi -\{k\phi \}) + 2\\
    &= (3-\phi )\{k\phi \} +2-\phi = \frac{\sqrt{5}}{\phi }\{k\phi \} + \frac{3-\sqrt{5}}{2}.
\end{align*}
In the second case, we have $\lfloor -\phi + (2\phi -1)\cdot \frac{\{k\phi \}}{\phi } \rfloor = -1$, so $\{c(k)\phi \} = \frac{\sqrt{5}}{\phi }\{k\phi \}+\frac{3-\sqrt{5}}{2}-1$, and the result follows.
\end{proof}

\begin{corollary}
$$\{c(n)\phi\} + \phi\{d(n)\phi\} =\left\{ \begin{array}{ll}
	   1,  &  \text{ if } \{n\phi\} > \frac{1}{\sqrt{5}},\\
	    2, & \text{ if } \{n\phi\} < \frac{1}{\sqrt{5}}.
	\end{array}\right.$$
\end{corollary}

\begin{theorem}

$\{s(k)\phi\} = \frac{\sqrt{5}}{2\phi}\{k\phi\}+b$, where $b\in \{0, \frac{1-\sqrt{5}}{4}, \frac{5-\sqrt{5}}{4}\}$ if $k$ is even, and $b \in \{-\frac{1}{2}, \frac{1}{2}, \frac{3-\sqrt{5}}{4}, 1-\frac{\sqrt{5}}{2}, 2-\frac{\sqrt{5}}{2}\}$ if $k$ is odd. In each case, the value of $b$ depends on the value of $\{k\phi\}$.

\end{theorem}

For the proof, we need to write $s(k)$ as a function of $c(k)$. Define $\Delta$ as the gap function of column $C$ such that
\[
\Delta (k) = c(k+1)-c(k) = \lfloor (k+1)\phi \rfloor - \lfloor k\phi \rfloor + 2.
\]
Then since $\lfloor (k+1)\phi \rfloor - \lfloor k\phi \rfloor \in (\phi - 1, \phi + 1)$ and $\Delta (k) \in \mathbb{N}$, we have $\Delta (k)$ = 3 or 4. So $c(k+1) - 1 > c(k) + 1$, which means $\{c(k)-1\} \cap \{c(k)+1\} = \emptyset.$ Now since $S = \{s(k)\} = \{c(k)-1\} \cup \{c(k)+1\}$, each two of the elements in $S$ are exactly generated by one element in $C$, and we have
$$s(k) = \left\{ \begin{array}{ll}
	   c(\frac{k}{2})+1,  &  \text{ if } k \text{ is even},\\[0.2em]
	    c(\frac{k+1}{2})-1, & \text{ if } k \text{ is odd}.
	\end{array}\right.$$

So we can see that to obtain the fractional properties of $\{s(k)\phi \}$, we need to investigate the fractional properties of $\{\frac{k}{2}\phi\}$ and $\{\frac{k+1}{2}\phi\}$, which can be written in the form of $\{k\phi \}$. These are established in the following lemma.
\begin{lemma}

When $k$ is an odd positive integer, $\{\frac{k+1}{2}\phi\} =  \frac{1}{2}\phi+\frac{1}{2}\{k\phi\}$ or $\frac{1}{2}(\phi-1)+\frac{1}{2}\{k\phi\}$ or  $\frac{1}{2}(\phi -2)+\frac{1}{2}\{k\phi\}$.
When $k$ is an even positive integer, $\{\frac{k}{2}\phi\} =  \frac{1}{2}\{k\phi\}$ or $\frac{1}{2}(1+\{k\phi\})$.
\end{lemma}

\begin{proof}

When $k$ is an odd positive integer, since $\{k\phi\} = \{\frac{k+1}{2}\phi+\frac{k+1}{2}\phi-\phi\} = \{2\{\frac{k+1}{2}\phi\}-\frac{\sqrt{5}-1}{2}\}$, we need to consider the following three cases:
\[
2\{\frac{k+1}{2}\phi\}-\frac{\sqrt{5}-1}{2} \in (1, 2) \text{ or } (0, 1) \text{ or } (-1, 0).
\]
In the first case, where $2\{\frac{k+1}{2}\phi\}-\frac{\sqrt{5}-1}{2} \in (1, 2)$, we have $\{\frac{k+1}{2}\phi\} = \frac{1}{2}\phi+\frac{1}{2}\{k\phi\}$. In the second case, where $2\{\frac{k+1}{2}\phi\}-\frac{\sqrt{5}-1}{2} \in (0, 1)$, we have $\{\frac{k+1}{2}\phi\} = \frac{1}{2}(\phi-1)+\frac{1}{2}\{k\phi\}$. In the third case, where $2\{\frac{k+1}{2}\phi\}-\frac{\sqrt{5}-1}{2} \in (-1, 0)$, we have $\{\frac{k+1}{2}\phi\} = \frac{1}{2}(\phi -2)+\frac{1}{2}\{k\phi\}$.

When $k$ is an even positive integer, $\{k\phi\} = \{\frac{k}{2}\phi + \frac{k}{2}\phi\} =  \{2\{\frac{k}{2}\phi\}\}$, and we need to consider the following two cases:
\[
\{\frac{k}{2}\phi\} \in (0, \frac{1}{2}), \text{ and } \{\frac{k}{2}\phi\} \in (
\frac{1}{2}, 1).
\]
In the first case, $\{\frac{k}{2}\phi\} = \frac{1}{2}\{k\phi\}$. In the second case,  $\{\frac{k}{2}\phi\} = \frac{1}{2}(1+\{k\phi\})$.
\end{proof}

Now we can proceed to the proof of Theorem 4.6. We need to consider two cases: when $k$ is an odd positive integer and when $k$ is an even positive integer.
\begin{case}
Let $k$ be even. Let $n$ be $\frac{k}{2}$. Using the KLM formula with $K = 1, L = 2, M = 0$, we have
\begin{align*}
    \{s(k)\phi\} &= \{(c(k/2)+1)\phi \} = (c(n)+1)\phi - a(c(n)+1) = (a(n)+2n)\phi - a(a(n)+2n)
    \\
    &= a(n)\phi + 2n\phi - b(n) - 2a(n) - \lfloor (2\phi -1)\cdot \{n\phi \}/\phi \rfloor \text{ } ((2\phi -1)\cdot \{n\phi \}/\phi = \{n\phi \}(3-\phi ))
    \\
    &= (n\phi - \{n\phi \})\phi + 2n\phi - (n\phi - \{n\phi \}) - n -2(n\phi - \{n\phi \})- \lfloor  \{n\phi \}(3-\phi ) \rfloor 
    \\
    &= -\{n\phi \}\phi + 3\{n\phi \} - (\{n\phi \}(3-\phi ) - \{\{n\phi \}(3-\phi )\})\\
    &= \{\{n\phi \}(3-\phi )\} = \{2\{n\phi \} + (1-\phi )\{n\phi \} \}
    \\
    &= \{\{2n\phi \}+(1-\phi )\{n\phi \} \} = \{\{k\phi \}+(1-\phi)\{\phi \cdot k/2 \} \}.
\end{align*}
From Lemma 4.7 we know that when $k$ is even, $\{\frac{k}{2}\phi\} =  \frac{1}{2}\{k\phi\} \text{ or } \frac{1}{2}(1+\{k\phi\})$.
\begin{subcase}
$\{\frac{k}{2}\phi\} = \frac{1}{2}\{k\phi\}$. Then 
\\$\{s(k)\phi\} = \{\{k\phi\}+(1-\phi)\frac{\{k\phi\}}{2}\} = \{\frac{5-\sqrt{5}}{4}\{k\phi\}+0\} = \frac{\sqrt{5}}{2\phi}\{k\phi\}+0$.
\end{subcase}
\begin{subcase}
$\{\frac{k}{2}\phi\}$ = $\frac{1}{2}(1+\{k\phi\})$. Then
\\$\{s(k)\phi\} = \{\{k\phi\}+(1-\phi)\frac{1+\{k\phi\}}{2}\} = \{\frac{5-\sqrt{5}}{4}\{k\phi\}+\frac{1-\sqrt{5}}{4}\}$. If $\{k\phi \} \in (\frac{1}{\sqrt{5}}, 1)$, then $\{s(k)\phi \}$ = $\frac{\sqrt{5}}{2\phi}\{k\phi \}+\frac{1-\sqrt{5}}{4}$. If $\{k\phi \} \in (0, \frac{1}{\sqrt{5}})$, then $\{s(k)\phi \}$ = $\frac{\sqrt{5}}{2\phi}\{k\phi \}+\frac{5-\sqrt{5}}{4}$.
\end{subcase}
To summarize, when $k$ is even, $\{s(k)\phi\} = \frac{\sqrt{5}}{2\phi}\{k\phi\}+b$, where $b\in \{0, \frac{1-\sqrt{5}}{4}, \frac{5-\sqrt{5}}{4}\}$.
\end{case}

\begin{case}
Let $k$ be odd. Let $n=\frac{k+1}{2}$, then by KLM formula where $K = 1, L = 2, M = -2$, we have
\begin{align*}
\{s(k)\phi\} &= \{c(\frac{k+1}{2})-1)\phi \} = \{(a(n)+2n-2)\phi \} = (a(n)+2n-2)\phi - a(a(n)+2n-2)\\
&= a(n)\phi + 2n\phi -2\phi - b(n) - 2a(n) - \lfloor -2\phi + (2\phi -1)\cdot \{n\phi \}/\phi \rfloor \\
&= a(n)\phi + 2n\phi -2\phi - b(n) - 2a(n) - \lfloor -2\phi + \{n\phi \}(3-\phi ) \rfloor \\
&= -\{n\phi \}\phi - 2\phi + 3\{n\phi \} - (-2\phi + \{n\phi \}(3-\phi ) - \{-2\phi + \{n\phi \}(3-\phi ) \})\\
&= \{-2\phi + \{n\phi \}(3-\phi ) \} = \{-2\phi + \{n\phi \}(2-1/\phi ) \} = \{-\phi + \{(2n-1)\phi \} - (1/\phi) \{n\phi \} \}\\
&= \{\frac{1}{\phi ^2} + \{k\phi \} - \frac{1}{\phi }\{\frac{k+1}{2}\phi \} \} = \{\frac{3-\sqrt{5}}{2} + \{k\phi \} + \frac{1-\sqrt{5}}{2}\{\frac{k+1}{2}\phi \} \}.
\end{align*}
From Lemma 4.7 we know that when $k$ is odd, $\{\frac{k+1}{2}\phi\} =  \frac{1}{2}\phi+\frac{1}{2}\{k\phi\} \text{ or }  \frac{1}{2}(\phi-1)+\frac{1}{2}\{k\phi\} \text{ or }  \frac{1}{2}(\phi -2)+\frac{1}{2}\{k\phi\}$.
\begin{subcase}
$\{\frac{k+1}{2}\phi\}$ = $\frac{1}{2}\phi+\frac{1}{2}\{k\phi\}$. Then \begin{align*}
    \{s(k)\phi \} &= \{\frac{3-\sqrt{5}}{2}+\{k\phi \} + \frac{1-\sqrt{5}}{2}(\frac{1}{2}\phi +\frac{1}{2}\{k\phi \}) \} \\
    &= \{ \frac{5-\sqrt{5}}{4}\{k\phi \}+\frac{3-\sqrt{5}}{2}-\frac{1}{2}\} = \{\frac{5-\sqrt{5}}{4}\{k\phi \}+1-\frac{\sqrt{5}}{2}\}.
\end{align*}
If $\frac{5-\sqrt{5}}{4}\{k\phi \}+1-\frac{\sqrt{5}}{2} > 0$, then $\{s(k)\phi\}$ = $\frac{\sqrt{5}}{2\phi }\{k\phi \}+1-\frac{\sqrt{5}}{2}.$ If $\frac{5-\sqrt{5}}{4}\{k\phi \}+1-\frac{\sqrt{5}}{2} < 0$, then $\{s(k)\phi\}$ = $\frac{\sqrt{5}}{2\phi }\{k\phi \}+2-\frac{\sqrt{5}}{2}.$
\end{subcase}
\begin{subcase}
$\{\frac{k+1}{2}\phi\}$ = $\frac{1}{2}(\phi-1)+\frac{1}{2}\{k\phi\}$. Then \begin{align*}
    \{s(k)\phi \} &= \{\frac{3-\sqrt{5}}{2}+\{k\phi \} + \frac{1-\sqrt{5}}{2}(\frac{1}{2}(\phi -1)+\frac{1}{2}\{k\phi \}) \} \\
    &= \{\frac{5-\sqrt{5}}{4}\{k\phi \}+\frac{3-\sqrt{5}}{2}+\frac{1-\sqrt{5}}{2}\cdot \frac{\sqrt{5}-1}{4}\} = \{\frac{5-\sqrt{5}}{4}\{k\phi \}+\frac{3-\sqrt{5}}{4}\} \\
    &= \frac{\sqrt{5}}{2\phi }\{k\phi \} + \frac{3-\sqrt{5}}{4}.
\end{align*}
\end{subcase}
\begin{subcase}
$\{\frac{k+1}{2}\phi\}$ = $\frac{1}{2}(\phi -2)+\frac{1}{2}\{k\phi\}$. Then
\begin{align*}
    \{s(k)\phi \} &= \{\frac{3-\sqrt{5}}{2}+\{k\phi \} + \frac{1-\sqrt{5}}{2}(\frac{1}{2}\{k\phi \}-\frac{2-\phi}{2}) \} \\
    &= \{\frac{5-\sqrt{5}}{4}\{k\phi \}+\frac{3-\sqrt{5}}{2}+\frac{\sqrt{5}-1}{2}\cdot \frac{3-\sqrt{5}}{4}\} = \{\frac{5-\sqrt{5}}{4}\{k\phi \}+\frac{1}{2}\}.
\end{align*}
If $\frac{5-\sqrt{5}}{4}\{k\phi \}+\frac{1}{2} < 1$, then $\{s(k)\phi \}$ = $\frac{\sqrt{5}}{2\phi }\{k\phi \}+\frac{1}{2}$. If $\frac{5-\sqrt{5}}{4}\{k\phi \}+\frac{1}{2} > 1$, then $\{s(k)\phi \}$ = $\frac{\sqrt{5}}{2\phi }\{k\phi \}-\frac{1}{2}$.
\end{subcase}
To summarize, when $k$ is odd, $\{s(k)\phi\} = \frac{\sqrt{5}}{2\phi}\{k\phi\}+b$, where $b \in \{-\frac{1}{2}, \frac{1}{2}, \frac{3-\sqrt{5}}{4}, 1-\frac{\sqrt{5}}{2}, 2-\frac{\sqrt{5}}{2}\}$. This proves the theorem.
\end{case}

\subsection{Main Theorem of Column Classifications}

Now with the above results of fractional parts identities, we can investigate how the elements of $S, C, D$ are distributed among the sets $A, B$. The following result illuminates this subject.

\begin{theorem}
By classifying the three columns $S, C, D$ of the 3-set partition into $A$ and $B$ numbers, we only have 6 out of 8 possible classifications for each row of $(S, C, D)$: 
\\$\{(A, A, A), (A, A, B), (A, B, A), (B, A, A), (B, B, A), (B, A, B)\}$.
\end{theorem}

\begin{proof}
\begin{table}[h]
\parbox{.48\linewidth}{
    \centering
    \begin{tabular}{c|c|c}
        $S$ & $C$ & $D$ \\ \hline
        1 & 2 & 4 \\
        3 & 6 & 11 \\
        5 & 9 & 15 \\
        7 & 13 & 22 \\
        8 & 17 & 29 \\
        10 & 20 & 33 \\
        \vdots & \vdots & \vdots
    \end{tabular}
    \caption{3-set Partition}
}
\hfill
\parbox{.48\linewidth}{
    \centering
    \begin{tabular}{c|c|c}
        $S$ & $C$ & $D$ \\ \hline
        $A$ & $B$ & $A$ \\
        $A$ & $A$ & $A$ \\
        $B$ & $A$ & $B$ \\
        $B$ & $B$ & $A$ \\
        $A$ & $A$ & $A$ \\
        $B$ & $B$ & $A$ \\
        \vdots & \vdots & \vdots
    \end{tabular}
    \caption{Column Classifications}
}
\end{table}

Table 3 above shows some examples of how columns $S, C, D$ in Table 2 are distributed among the $A$ and $B$ numbers.

We first investigate how the elements in column $C, D$ are distributed among column $A$ and column $B$. From Theorem 3.2 we know that $\{a(n)\phi\} = 1-\frac{\{n\phi \}}{\phi} > 1-\frac{1}{\phi } = \frac{1}{\phi ^2}, \{b(n)\phi \} = \frac{\{n\phi \}}{\phi ^2} < \frac{1}{\phi ^2}.$ So by Theorem 4.3 and Theorem 4.4, we have the following results:
\\$c(k) \in A$ means $\{c(k)\phi\} > \frac{1}{\phi^2}$ if and only if
    $$ \{c(k)\phi \} = \left\{ \begin{array}{ll}
	   \frac{\sqrt{5}}{\phi}\{k\phi\}+\frac{1-\sqrt{5}}{2} > \frac{1}{\phi^2},  &  \text{ when } \{k\phi \} > \frac{1}{\sqrt{5}},\\
	    \frac{\sqrt{5}}{\phi}\{k\phi\}+\frac{3-\sqrt{5}}{2} > \frac{1}{\phi^2}, & \text{ when } \{k\phi \} < \frac{1}{\sqrt{5}}.
	\end{array}\right.$$
\\This means that in the first case $\{k\phi \} > \frac{5+\sqrt{5}}{10}$ while in the second case $\{k\phi \} < \frac{1}{\sqrt{5}}$.
\\$c(k) \in B$ means $\{c(k)\phi\} < \frac{1}{\phi^2}$ if and only if
    $$ \{c(k)\phi \} = \left\{ \begin{array}{ll}
	   \frac{\sqrt{5}}{\phi}\{k\phi\}+\frac{1-\sqrt{5}}{2} < \frac{1}{\phi^2},  &  \text{ when } \{k\phi \} > \frac{1}{\sqrt{5}},\\
	    \frac{\sqrt{5}}{\phi}\{k\phi\}+\frac{3-\sqrt{5}}{2} < \frac{1}{\phi^2}, & \text{ when } \{k\phi \} < \frac{1}{\sqrt{5}}.
	\end{array}\right.$$
\\This means that we have $\frac{1}{\sqrt{5}} < \{k\phi \} < \frac{5+\sqrt{5}}{10}$.
\\$d(k) \in A$ means $\{d(k)\phi\} > \frac{1}{\phi^2} \Longleftrightarrow \{d(k)\phi \} = 1-\frac{\sqrt{5}}{\phi^2}\{k\phi \} > \frac{1}{\phi^2} \Longleftrightarrow \{k\phi \} < \frac{5+\sqrt{5}}{10}$.
\\$d(k) \in B$ means $\{d(k)\phi\} < \frac{1}{\phi^2} \Longleftrightarrow \{d(k)\phi \} = 1-\frac{\sqrt{5}}{\phi^2}\{k\phi \} < \frac{1}{\phi^2} \Longleftrightarrow \{k\phi \} > \frac{5+\sqrt{5}}{10}$.
\vspace{\baselineskip}

Combining the results of classifications for column $C$ and $D$, we can conclude that $c(k) \in A \text{ and } d(k) \in A \Longleftrightarrow \{k\phi \} \in (0, \frac{1}{\sqrt{5}})$, $c(k) \in A \text{ and } d(k) \in B \Longleftrightarrow \{k\phi \} \in (\frac{5+\sqrt{5}}{10}, 1)$, $c(k) \in B \text{ and } d(k) \in A \Longleftrightarrow \{k\phi \} \in (\frac{1}{\sqrt{5}}, \frac{5+\sqrt{5}}{10})$, $c(k) \in B \text{ and } d(k) \in B \Longleftrightarrow \{k\phi \} \in (\frac{1}{\sqrt{5}}, \frac{5+\sqrt{5}}{10}) \text{ and } \{k\phi \} \in (\frac{5+\sqrt{5}}{10}, 1)$. In the last case, there is no $k \in \mathbb{N}$ such that $(c(k), d(k)) \in (B \times B)$.
\vspace{\baselineskip}

So $(c(k), d(k)) \in (A \times A) \cup (A \times B) \cup (B \times A)$, and $s(k)$ can be either in $A$ or in $B$. (The density of the set of integers $k$ such that $s(k) \in A$ remains unknown). Hence, we have
\begin{align*}
(s(k),c(k),d(k)) &\in (A \times A \times A) \cup (A \times A \times B) \cup (A \times B \times A)\\
&\cup (B \times A \times A) \cup (B \times B \times A) \cup (B \times A \times B).
\end{align*}
This proves the theorem.
\end{proof}

\begin{remark}
If we instead investigate how the elements of $A$ and $B$ are distributed among the $S, C, D$ numbers, we can also observe that there are only 5 out of 9 cases occurring. Our numerical data suggests the following density values:
\begin{center}
    \begin{tabular}{c|ccccccccc}
    
        Pair &    $SC$ & $CS$ & $DS$ & $CD$ & $SS$ & $DC$ & $SS$ & $CC$ & $DD$\\\hline
        Density & $\frac{1}{5}$ & $\frac{1}{5}$ & $\frac{1}{5}$ & $\frac{\phi-1}{5}$ & $\frac{3-\phi}{5}$ & $0$ & $0$ & $0$ & $0$
    \end{tabular}
    
    \end{center}
The exact reasoning of the distribution of these 9 pairs remains an open problem, but it is of interest to further investigate such densities.
\end{remark}

\section{Acknowledgements}
The authors discovered these results while participating in an Illinois Geometry Lab group
administered by Professors A.J.Hildebrand and K.B. Stolarsky. We thank them for advice on 
format and exposition. We also thank Professor Stolarsky for providing us with many of the references.

\end{document}